\documentclass{article}
\usepackage[utf8]{inputenc}
\usepackage{amsmath,amssymb,amsfonts,fullpage}
\usepackage{amsthm}
\usepackage{graphics}
\usepackage{graphicx}
\usepackage{hyperref}
\usepackage{color}
\usepackage{diagrams}

\newcommand{\N}{\mathbb{N}}
\newcommand{\R}{\mathbb{R}}
\newcommand{\Z}{\mathbb{Z}}

\newcommand{\half}{\tfrac{1}{2}}

\newtheorem{theorem}{Theorem}[section]
\newtheorem{proposition}[theorem]{Proposition}
\newtheorem{lemma}[theorem]{Lemma}

\title{An interacting particle system with geometric jump rates near a partially reflecting boundary}
\author{Jeffrey Kuan}

\begin{document}

\maketitle

\abstract{This paper constructs a new interacting particle system on $\mathbb{N}\times\mathbb{Z}_+$ with geometric jumps near the boundary $\{0\}\times\mathbb{Z}_+$ which partially reflects the particles. The projection to each horizontal level is Markov, and on every level the dynamics match stochastic matrices constructed from pure alpha characters of $Sp(\infty)$, while on every other level they match an interacting particle system from Pieri formulas for $Sp(2r)$. Using a previously discovered correlation kernel, asymptotics are shown to be the Discrete Jacobi and Symmetric Pearcey processes.}

\section{Introduction}
To motivate this paper, first review some previous results. In \cite{WW}, the authors construct a continuous--time interacting particle system on $\mathbb{N} \times \mathbb{Z}_+$ using the representation theory of symplectic Lie groups. One distinguishing feature of these dynamics is a wall at $\{0\} \times \mathbb{Z}_+$ which suppresses jumps of particles into the wall. In \cite{D2}, the Pieri formulas from the representation theory of the symplectic Lie groups $Sp(2r)$ are used to construct discrete--time dynamics with geometric jumps, and again jumps into the wall are suppressed.  In \cite{C}, there is a construction of continuous--time dynamics using Plancherel characters of the infinite--dimensional symplectic group $Sp(\infty)$, and again there is a suppressing wall. 

Some previous work had been done with the orthogonal groups as well. In \cite{BK}, Plancherel characters of the infinite--dimensional orthogonal group $O(\infty)$ led to continuous--time dynamics with a  \textit{reflecting} wall, and in \cite{D1}, Pieri rules for the orthogonal groups $O(2r),O(2r+1)$ led to interacting particles with discrete--time geometric jumps, again with a reflecting wall. In \cite{K}, it was shown that the dynamics of \cite{D1} on each level $\N\times\{k\}$ fit into the general framework of \cite{BK} with pure alpha characters of $O(\infty)$. Therefore, it is reasonable to expect that the dynamics of \cite{D2} might also fit into the framework of \cite{C} with pure alpha characters of $Sp(\infty)$. However, it turns out that the dynamics only match on the even levels $\mathbb{N} \times \{2r\}$.

In order to create a physically meaningful interacting particle system which matches that of \cite{C} on every level, we will slightly modify the Pieri formulas of \cite{D1},\cite{D2}. The result is a wall which is \textit{partially} reflecting. Mathematically, this means that a jump to $-x$ is reflected to $x-1$, rather than being totally reflected to $x$ or totally suppressed at $0$. Observe that after the usual scaling limit of discrete--time geometric jumps to continuous--time jumps with exponential waiting times, the particles only jump one step, so the partially reflecting boundary becomes a suppressing boundary.

Note that there may be an algebraic intuition for the discrepancy between the dynamics of \cite{D2} and \cite{C}. The odd symplectic groups $Sp(2r+1)$ of \cite{P} are not simple, in contrast to $Sp(2r),O(2r),O(2r+1)$. In a sense, the odd symplectic groups are less canonical, which may explain why the two dynamics only match at the levels corresponding to $Sp(2r)$. 

The paper is outlined as follows. In section \ref{IPS}, the interacting particle system is defined. In section \ref{PureAlpha}, new formulas for the stochastic matrices from \cite{C} are written. In section  \ref{Projection}, it is shown that the projection to each horizontal level is still Markov, and the resulting transition probabilities are precisely the ones from section \ref{PureAlpha}. Using the explicit expression for the correlation kernel in \cite{C}, section \ref{Asymp} finds the asymptotics for our particle system. 

Note that despite the algebraic motivation and background, the body of the paper is written with minimal reference to representation theory.

\textbf{Acknowledgments}. 
Financial support was available through NSF grant DMS--1502665.

\section{Interacting Particle System}\label{IPS}

First define the state space for the interacting particles. For $k\geq 1$, define
$$
\mathcal{W}_k = \{ (\lambda_1 \geq \ldots \geq \lambda_r) : \lambda_i \in \mathbb{N} \}, \text{ where }  r = \lfloor \tfrac{k+1}{2} \rfloor
$$
If $\lambda=(\lambda_1 \geq \ldots \geq \lambda_r)$ and $\mu=(\mu_1 \geq \ldots \geq \mu_r)$ or $\mu=(\mu_1 \geq \ldots \geq \mu_{r+1})$, say that $\lambda\prec\mu$ if $\mu_{i+1} \leq \lambda_i$ and $\lambda_i \leq \mu_i $ for all possible values of $i$. For $k\leq l$ let
$$
\mathcal{W}_{k,l} := \{ (\lambda^{(k)} \prec \lambda^{(k+1)} \prec \ldots \prec \lambda^{(l)}) : \lambda^{(j)} \in \mathcal{W}_j\}.
$$
The state space for the interacting particles will be
$$
\mathcal{W} = \{ (\lambda^{(1)} \prec \lambda^{(2)} \prec \ldots ) : \lambda^{(j)} \in \mathcal{W}_j \text{ for } 1 \leq j < \infty\}.
$$

If $\xi_1$ and $\xi_2$ are two independent geometric random variables with parameter $q$ (i.e. $\mathbb{P}(\xi_i = x)=(1-q)q^x$ for $x\geq 0$), then for any $x,y\geq 0$
\begin{align*}
\mathbb{P}(x + \xi_1 - \xi_2 = y) &= \sum_{c=0}^{\infty} \mathbb{P}(\xi_2 = c)\mathbb{P}(\xi_1 = y - x + c) \\
&= (1-q)^2 \sum_{c= \mathrm{min}(x-y,0)}^{\infty} q^c q^{y-x+c}\\
&= \frac{1-q}{1+q} q^{| x-y |}
\end{align*}
Let $\{x\}$ denote the modified absolute value
$$
\{x\}=
\begin{cases}
x, x\geq 0 \\
-x-1, x<0
\end{cases}
$$
Thus
$$
\mathbb{P}(\{x + \xi_1 - \xi_2\} = y) = \frac{1-q}{1+q} (q^{|x-y|} + q^{x+y+1})
$$
With this in mind, define
$$
R(x,y) = \frac{1-q}{1+q} (q^{|x-y|} + q^{x+y+1}).
$$
Observe that if $x\geq y\geq z\geq 0$, then
\begin{equation}\label{Req}
R(x,z) = q^{x-y}R(y,z) 
\end{equation}

The particles live on the lattice $\N\times\Z_+$ where $\N$ denotes the non--negative integers and $\Z_+$ denotes the positive integers. The horizontal line $\N\times\{k\}$ is often called the $k$th level. There are always $\lfloor \tfrac{k+1}{2} \rfloor$ particles on the $k$th level, whose positions at time $n$ will be denoted $X^k_1(n)\geq X^k_2(n)\geq X^k_3(n)\geq \ldots\geq X^k_{\lfloor (k+1)/2 \rfloor}(n)\geq 0$. The time can take integer or half--integer values. For convenience of notation, $X^k(n)$ will denote the element $(X^k_1(n), X^k_2(n), X^k_3(n), \ldots, X^k_{\lfloor (k+1)/2 \rfloor}(n))$ $\in\mathbb{N}^{\lfloor (k+1)/2 \rfloor}$. More than one particle may occupy a lattice point. The particles must satisfy the \textit{interlacing property}
\begin{equation*}
X^{k+1}_{i+1}(n)  \leq X^k_i(n) \leq X^{k+1}_i(n)
\end{equation*}
for all meaningful values of $k$ and $i$. This will be denoted $X^k\prec X^{k+1}$. With this notation, the state space can be described as the set of all sequences $(X^1\prec X^2 \prec \ldots)$ where each $X^k\in \mathbb{N}^{\lfloor (k+1)/2 \rfloor}$. The initial condition is $X_i^k(0)=0$, called the \textit{densely packed} initial conditions. Now let us describe the dynamics. 

For $n\geq 0,k\geq 1$ and $1\leq i \leq \lfloor \tfrac{k+1}{2} \rfloor$, define random variables
\[
\xi^k_i(n+1/2), \ \ \xi^k_i(n)
\]
which are independent identically distributed geometric random variables with parameter $q$. In other words, $\mathbb{P}(\xi^1_1(1/2)=x)=q^x(1-q)$ for $x\in\N$. 

At time $n$, all the particles except $X^k_{({k+1})/{2}}(n)$ try to jump to the left one after another in such a way that the interlacing property is preserved. The particles $X^k_{({k+1})/{2}}(n)$ do not jump on their own. The precise definition is 
\begin{eqnarray*}
X^k_{(k+1)/{2}}(n+\tfrac{1}{2})&=&\min(X^k_{(k+1)/{2}}(n), X_{(k-1)/2}^{k-1}(n+\tfrac{1}{2}))\ \ k \text{ odd}\\
X^k_i(n+\tfrac{1}{2})&=& \max(X^{k-1}_i(n), \min(X^k_i(n),X^{k-1}_{i-1}(n+\tfrac{1}{2})) - \xi^k_i(n+\tfrac{1}{2})),
\end{eqnarray*}
where $X^{k-1}_0(n+\tfrac{1}{2})$ is formally set to $+\infty$.

At time $n+\tfrac{1}{2}$, all the particles except $X^k_{({k+1})/{2}}(n+\half)$ try to jump to the right one after another in such a way that the interlacing property is preserved. The particles $X^k_{({k+1})/{2}}(n+\half)$ jump according to the law $R$. The precise definition is 
$$
X^k_{(k+1)/{2}}(n+1)=\min(\{ X^k_{(k+1)/{2}}(n) + \xi^k_{(k+1)/2}(n+1) - \xi^k_{(k+1)/2}(n+\half) \}, X^{k-1}_{(k-1)/{2}}(n) )
$$
when $k$ is odd and
$$
X^k_i(n+1)= \min(X^{k-1}_{i-1}(n+\half), \max(X^k_i(n+\half),X^{k-1}_{i}(n+1)) + \xi^k_i(n+1)),
$$
where $X^{k-1}_0(n+1)$ is formally set to $+\infty$.

Let us explain the particle system. The particles preserve the interlacing property in two ways: by pushing particles above it, and being blocked by particles below it. So, for example, in the left jumps, the expression $\min(X^k_i(n),X^{k-1}_{i-1}(n+\tfrac{1}{2}))$ represents the location of the particle after it has been pushed by a particle below and to the right. Then the particle attempts to jump to the left, so the term $\xi^k_i(n+\tfrac{1}{2})$ is subtracted. However, the particle may be blocked a particle below and to the left, so we must take the maximum with $X^{k-1}_i(n)$.

While $X_i^k(n)$ is not simple, applying the shift $\tilde{X}_i^k(n)=X_i^k(n) + \lfloor\tfrac{k+1}{2}\rfloor -i$ yields a simple process. In other words, $\tilde{X}$ can only have one particle at each location.

Figure \ref{Jumping} shows an example of $\tilde{X}$. 

\section{Relation to pure alpha characters of $Sp(\infty)$}\label{PureAlpha}
For $k\geq 1$, define $r_k = \lfloor \frac{k+1}{2}\rfloor$ and $\tilde{\lambda}_i = \lambda_i + r_k - i$ and 
$$
a_k:=
\begin{cases}
-1/2, & \text{if } k \text{ odd}\\
1/2, & \text{if } k \text{ even}
\end{cases}
$$ 
Note that $2r_k+a_k-1/2=k$. Set
$$
\langle f,g\rangle_{a} = \frac{2^{a+1/2}}{\pi}\int_{-1}^1 f(x)g(x) (1-x)^{a}(1+x)^{1/2}dx
$$

There are explicit functions $s_k$: for $k=2r$ or $k=2r+1$,
$$
s_{2r}(\lambda) = \prod_{1\leq i<j\leq r} \frac{l_i^2-l_j^2}{m_i^2-m_j^2} \cdot \prod_{1\leq i\leq r} \frac{l_i}{m_i}, \quad s_{2r-1}(\lambda) = \prod_{1\leq i<j\leq r} \frac{l_i'^2-l_j'^2}{m_i'^2-m_j'^2} 
$$
where $l_i = \lambda_i + (r_k-i)+1, m_i = (r_k-i) +1,$ and $l_i' = l_i - 1/2, m_i' = m_i-1/2$ for $1\leq i\leq r$. These satisfy
\begin{equation}\label{Central}
\sum_{ \substack{ \lambda \in \mathcal{W}_{k-1} \\ \lambda \prec \mu}} s_{k-1}(\lambda) = s_k(\mu).
\end{equation}
Note that $s_1(\lambda)=1$ for all $\lambda$ and $s_2(\lambda) = \lambda_1+1$.

The following transition probabilities on $\mathcal{W}_n$ are from section 5.1 of  \cite{C} are
$$
T_k^{\phi}(\lambda,\mu):= \det\left[ \left \langle J_{\tilde{\lambda}_i,a_k}, J_{\tilde{\mu}_j,a_k}\phi \right\rangle \right]_{i,j=1}^{r_k} \frac{  s_k(\mu) }{ s_k(\lambda)}. 
$$
where $J$ are the Jacobi polynomials satisfying 
$$
J_{k,1/2}\left( \frac{z+z^{-1}}{2}\right) = \frac{z^{k+1}-z^{-(k+1)}}{z-z^{-1}}, \quad J_{k,-1/2}\left( \frac{z+z^{-1}}{2}\right) = \frac{z^{k+1/2}+z^{-(k+1/2)}}{z^{1/2}+z^{-1/2}}.
$$
When $\phi(x) = e^{t(x-1)}$, these are the transition probabilities arising from Plancherel characters of $Sp(\infty)$. In this paper a different $\phi(x)$ depending on a parameter $\alpha \geq 0$ will be considered. Note that for general $\phi(x)$, a priori there is not an obvious physical description of the dynamics. 

The following formula is standard (for example, see Lemma 3.3 from \cite{K}). Let  $c=(c_1\geq c_2\geq \ldots \geq c_r)$ and $\lambda=(\lambda_1\geq \lambda_2 \geq \ldots \geq \lambda_r)$. Set  
\[
\psi(s,l)=
\begin{cases}
1, \ \ \text{if}\ l\geq s\\
0, \ \ \text{if}\ l<s.
\end{cases}
\]
Then 
\begin{equation}\label{Interlacing}
\det[\psi(c_i-i+r,\lambda_j-j+r)]_1^r=
\begin{cases}
1, \ \ \text{if}\ c\prec\lambda,\\
0, \ \ \text{if}\ c\not\prec\lambda.
\end{cases}
\end{equation}

Set
\begin{align*}
P_{2r+1}(\lambda,\beta) &= \sum_{c \in \mathbb{N}^r, c\prec \lambda,\beta} (1-q)^{2r} \frac{s_{2r+1}(\beta)}{s_{2r+1}(\lambda)} q^{\sum_{i=1}^r \lambda_i + \beta_i -2c_i} R(\lambda_{r+1},\beta_{r+1})\\
P_{2r}(\lambda,\beta) &= \sum_{c \in \mathbb{N}^r, c\prec \lambda,\beta} (1-q)^{2r} \frac{s_{2r}(\beta)}{s_{2r}(\lambda)} q^{\sum_{i=1}^r \lambda_i + \beta_i -2c_i} 
\end{align*}
Note that the formula for $P_{2r}$ is essentially identical to (2) from \cite{D2}, and the formula for $P_{2r+1}$ is also similar to a related formula from \cite{D1} with a different definition of $R$ (see Proposition 5.2 and Theorem 7.1). The discussions in \cite{D1},\cite{D2} describe how to obtain these formulas from Pieri's rule.

This next proposition shows that $P$ and $T$ are the same. The proof is similar to the Proposition 3.1 from \cite{K}. The only difference is that $R(\cdot,\cdot)$ has a different definition here, but it turns out that the only relevant information about $R$ for the proof is that \eqref{Req} is true. The full proof is still included here for completeness, because much of the notation is different.

\begin{proposition}\label{Same}
Let
$$
\phi(x) = \frac{1}{1+\alpha(1-x)+\frac{\alpha^2}{2}(1-x)}, \quad \alpha=\frac{2q}{1-q}.
$$
Then
$$
P_{2r+1}=T_{2r+1}^{\phi}, \quad P_{2r}=T_{2r}^{\phi} 
$$
\end{proposition}
\begin{proof}
We first prove this for $r=1$. Substituting $x=(z+z^{-1})/2$
$$
\phi(x) = \frac{(1-q)^2}{(z-q)(1-qz)}, \quad \quad (1-x)^{a}(1+x)^{1/2}dx \mapsto
\begin{cases}
\frac{(z^{1/2}+z^{-1/2})^2}{4iz}, & a=-1/2\\
-\frac{(z-z^{-1})^2}{8iz}, & a=1/2
\end{cases}
$$
and the integral over $[-1,1]$ becomes an integral over the unit circle, with an extra factor of $1/2$ occurring because the map $z\mapsto x$ is two--to--one.
Thus, the term inside the $1\times 1$ determinant in $T^{\phi}_n$ can be calculated from the identities
\begin{align*}
\frac{1}{4\pi i}\oint (z^{k+1/2}+z^{-(k+1/2)}) (z^{l+1/2}+z^{-(l+1/2)}) \frac{(1-q)^2 dz}{(z-q)(1-qz)} &= \frac{1-q}{1+q} (q^{k+l+1} + q^{| k-l|} )\\
-\frac{1}{4\pi i}\oint (z^{k+1}-z^{-k-1}) (z^{l+1}-z^{-l-1}) \frac{(1-q)^2dz}{(z-q)(1-qz)} &= \frac{1-q}{1+q} (q^{| k-l|}-q^{k+l+2}  ) 
\end{align*}
for $n=1,2$ respectively. The first line is $R(k,l)=P_1(k,l)$. For the second line, note that
\begin{equation}\label{Above}
(1-q)^2\sum_{c=0}^{\text{min}(\lambda,\beta)} q^{\lambda+\beta-2c} = (1-q)^2\frac{q^{\lambda+\beta}-q^{\vert \lambda-\beta\vert -2}}{1-q^{-2}} = \frac{1-q}{1+q}(q^{\vert \lambda-\beta\vert }-q^{\lambda+\beta+2})
\end{equation}
which shows that $P_2 = T_2^{\phi}$.

Now proceed to higher values of $r$. By  \eqref{Interlacing},
$$
P_{2r}(\lambda,\beta) = (1-q)^{2r} \frac{s_{2r}(\beta)}{s_{2r}(\lambda)} \sum_{s_1>\ldots>s_
r\geq 0} \det[f_{s_i}(\lambda_j-j+r)] \det[f_{s_i}(\beta_j-j+r)]
$$
where
$$
f_{s}(l)= q^{l-s}\psi(s,l).
$$
By Lemma 2.1 of \cite{BK}, this equals
$$
(1-q)^{2r} \frac{s_{2r}(\beta)}{s_{2r}(\lambda)} \det\left[\sum_{s=0}^{\infty}f_{s}(\lambda_i-i+r)f_{s}(\beta_j-j+r)\right].
$$
By identity \eqref{Above}, $P_{2r}=T_{2r}^{\phi}$.

For $P_{2r+1}$, start with the following claim: if $\mathrm{max}(\lambda_{r+1},\beta_{r+1}) > \mathrm{min}(\lambda_{r},\beta_{r})$ then $P_{2r+1}(\lambda,\beta) =T_{2r+1}^{\psi}(\lambda,\beta)=0$. To see this claim, first notice that $P_{2r+1}=0$ follows immediately from the description of the interacting particle system, or from the fact that $\{c \in \mathbb{N}^{r}: c\prec\lambda,\beta\}$ is empty. By \eqref{Req}, in the matrix of $T_{2r+1}$ the $r$th column is a multiple of the $(r+1)$th column, so that $T_{2r+1}=0$.

Because of this claim, assume that $\mathrm{max}(\lambda_r,\beta_r) \leq \mathrm{min}(\lambda_{r-1},\beta_{r-1})$.

Lemma 2.1 from \cite{BK} is not immediately applicable, because we are summing over elements of $\mathbb{N}^r$ while the determinants are of size $r+1$. Notice, however, that $c\prec\lambda,\beta$ if and only if $c\prec\lambda_{\text{red}},\beta_{\text{red}}$ (where $\lambda_{\text{red}},\beta_{\text{red}}$ denote $(\lambda_1,\ldots,\lambda_{r}),(\beta_1,\ldots,\beta_{r})$) and $c_{r}\geq\max(\lambda_{r+1},\beta_{r+1})$. Thus
\begin{multline*}
P_{2r+1}(\lambda,\beta) = \sum_{c \in \mathbb{N}^r, c\prec \lambda,\beta} (1-q)^{2r} \frac{s_{2r}(\beta)}{s_{2r}(\lambda)} q^{\sum_{i=1}^r \lambda_i + \beta_i -2c_i} R(\lambda_{r+1},\beta_{r+1})\\
=(1-q)^{2r}\frac{s_{2r}(\beta)}{s_{2r}(\lambda)} R(\lambda_{r+1},\beta_{r+1})   \\
\times\sum_{s_1>s_2>\ldots> s_{r}\geq\max(\lambda_{r+1},\beta_{r+1})} \det[f_{s_i,1}(\lambda_j-j+r)]_1^{r}\det[f_{s_i,1}(\beta_j-j+r)]_1^{r}\\
=(1-q)^{2r}R(\lambda_{r+1},\beta_{r+1})  \frac{s_{2r}(\beta)}{s_{2r}(\lambda)}\det\left[\sum_{s=\max(\lambda_{r+1},\beta_{r+1})}^{\infty}f_{s,1}(\lambda_i-i+r)f_{s,1}(\beta_j-j+r)\right]_1^{r}.
\end{multline*}
A straightforward calculation shows that if $\max(\lambda_{r+1},\beta_{r+1})\leq\min(x,y)$, then
\[
\sum_{s=\max(\lambda_{r+1},\beta_{r+1})}^{\infty}f_{k,1}(x)f_{k,1}(y)=\frac{q^{x+y-2\max(\lambda_{r+1},\beta_{r+1})+2}-q^{\vert x-y\vert}}{q^2-1}  .
\]

Thus it remains to show that
\begin{multline*}
 \det\left[ R(\lambda_i-i+r+1,\beta_j-j+r+1) \right]_{i,j=1}^{r+1} \\
 =(1-q)^{2r}R(\lambda_{r+1},\beta_{r+1}) \det\left[\frac{q^{x_i+y_j-2\max(\lambda_{r+1},\beta_{r+1})+2}-q^{\vert x_i-y_j\vert}}{q^2-1}   \right]_1^r
\end{multline*}
where $x_i = \lambda_i-i+r$ and $y_j = \beta_j-j+r$. Recall that
$$
R(x,y) = \frac{1-q}{1+q} (q^{|x-y|} + q^{x+y+1}).
$$

To show that this is true, perform a sequence of operations to the smaller matrix. These operations are slightly different for $\lambda_{r+1}>\beta_{r+1}$ and $\lambda_{r+1}\leq\beta_{r+1}$. Consider $\lambda_r>\beta_r$ for now.

First, add a row and a column to the matrix of size $r$. The $(r+1)$th column is $[0,0,0,\ldots,0,R(\lambda_{r+1},\beta_{r+1})]$ and the $(r+1)$th row is $[R(\lambda_{r+1},\beta_1+r),R(\lambda_{r+1},\beta_2-1+r),\ldots,R(\lambda_{r+1},\beta_r+1),R(\lambda_{r+1},\beta_{r+1})]$.
This multiplies the determinant by $R(\lambda_{r+1},\beta_{r+1})$.

Second, for $1\leq i\leq r$, perform row operations by replacing the $i$th row with
\[
i\text{th row} + \frac{1}{(q-1)^2}\frac{R(\lambda_i-i+r+1,\beta_{r+1})}{R(\lambda_{r+1},\beta_{r+1})}((r+1)\text{th row}).
\]
For $1\leq j\leq r$ and letting $(x,y)=(\lambda_i-i+r+1,\beta_j-j+r+1)$, the $(i,j)$ entry is (recall \eqref{Req})
\begin{eqnarray*}
&&\frac{q^{x+y-2\max(\lambda_{r+1},\beta_{r+1})}-q^{\vert x-y\vert}}{q^2-1}+ \frac{1}{(q-1)^2}\frac{R(x,\beta_{r+1})}{R(\lambda_{r+1},\beta_{r+1})}R(\lambda_{r+1},y)\\
&=&\frac{q^{x+y-2\max(\lambda_{r+1},\beta_{r+1})}-q^{\vert x-y\vert}}{q^2-1}-\dfrac{q^{x-\lambda_{r+1}}}{q^2-1}( q^{\lambda_{r+1}+y+1} + q^{\vert \lambda_{r+1}-y\vert})\\
&=&\frac{-q^{x+y+1}-q^{\vert x-y\vert}}{q^2-1}=(1-q)^{-2}R(x,y).
\end{eqnarray*}
Here, we used the fact that $y\geq\beta_{r}\geq\min(\lambda_{r},\beta_{r})\geq\lambda_{r+1}$ and $y\geq\lambda_{r+1}>\beta_{r+1}\geq 0$. For $j=r+1$, $\lambda_{r+1}>y=\beta_{r+1}$, so the $(i,j)$ entry is
$$
0+ \frac{1}{(q-1)^2}\frac{R(\lambda_i-i+r+1,\beta_{r+1})}{R(\lambda_{r+1},\beta_{r+1})}R(\lambda_{r+1},y)  =(1-q)^{-2}R(x,y).
$$
Thus, the larger determinant is $(1-q)^{2r}R(\lambda_r,\beta_r)$ times the larger determinant.

Now consider $\lambda_{r+1}\leq\beta_{r+1}$. First, add the $(r+1)$th row, which is equal to $[0,0,\ldots,0,R(\lambda_{r+1},\beta_{r+1})]$, and add the $(r+1)$th column which is $[R(\lambda_1-1+r,\beta_{r+1}),R(\lambda_2-2+r,\beta_{r+1}),\ldots,R(\lambda_{r+1},\beta_{r+1})]$. This multiplies the determinant by $R(\lambda_{r+1},\beta_{r+1})$. 

Second, for $1\leq j\leq r$, perform column operations by replacing the $j$th column with
\[
j\text{th column} + \frac{1}{(q-1)^2}\frac{R(\lambda_{r+1},\beta_j-j+r+1)}{R(\lambda_{r+1},\beta_{r+1})}((r+1)\text{th column}).
\]
Once again, this yields a matrix whose entries are $(1-q)^{-2}R(\lambda_i-i+r+1,\beta_j-j+r+1)$, except for the last column, which is $R(\lambda_i-i+r+1,\beta_{r+1})$. 
\end{proof}

Note that while the projections to each level are the same, the multi--level dynamics are different. For the dynamics in this paper, there is zero probability of a jump from $(0) \prec (1) \prec (1,0)$ to $(0) \prec (0) \prec (0,0)$ on the bottom three levels, because $X_1^2$ prevents $X_1^3$ from jumping to $0$. However, in the dynamics of \cite{C}, this probability is nonzero because all of the terms in (60) are nonzero.

\section{Projections to levels}\label{Projection}
This section will provide a proof that the projection to each level is Markov with an explicit expression for the Markov operator. Note that the method of the proof is very similar to that of \cite{D1,D2}. The primary difference is that due to the different expression for $R$ and for the branching rule, the identities \eqref{keyidentity1} and \eqref{keyidentity4} are changed.

We consider the subset $\mathcal{W}^{(2)}_{k}$  of $\mathcal{W}_{k}\times \mathcal{W}_k$ defined by
$$
\mathcal{W}^{(2)}_{k}=\{(z,y): z\prec y\}\subseteq \mathcal{W}_{k}\times \mathcal{W}_k,
$$ 
and  define  a Markov kernel $S_k$ on $\mathcal{W}^{(2)}_k$ by 
\begin{align}\label{Skeven}
S_{k}((z,y),(z',y'))&=(1-q)^{k}\frac{s_{k}(y')}{s_{k}(y)}q^{\sum_{i=1}^{r}(y_i+y'_i-2z'_i)}1_{ z'\prec y,y'}
\end{align} 
when $k=2r$, and  
\begin{align}\label{Skodd}
S_{k}((z,y),(z',y'))&=(1-q)^{k-1}\frac{s_{k}(y')}{s_{k}(y)}R(y_r,y'_r)q^{\sum_{i=1}^{r-1}(y_i+y'_i-2z'_i)} 1_{z'_r=y_r}1_{ z'\prec y,y'} 
\end{align}
when $k=2r-1$. 
Since the expression for $S_k$ does not depend on $z$, also write it as $S_k(y,(z',y'))$. Note that
\begin{equation}\label{SP}
\sum_{z' \in \mathcal{W}^{(2)}_k} S_k((z,y),(z',y')) = P_k(y,y')
\end{equation}
Thus Proposition \ref{Same} implies that $S_k$ is a Markov operator.

\begin{theorem}\label{MarkovProj}
For each $k\geq 1$, the random process $(X^k(n-\half),X^k(n))_{n \in \mathbb{N}}$ is a Markov process with transition kernel given by $S_k$. Furthermore, $(X^k(n))_{n\in \mathbb{N}}$ is a Markov process with transition kernel given by $P_k$.
\end{theorem}
\begin{proof}
It suffices to prove the first statement, because by \eqref{SP} the second follows from the first.

The proof will follow from induction on $k$. For $k=1$, Theorem \ref{MarkovProj} is clearly true. If the random process $(X^{k-1}(n-\half),X^{k-1}(n))_{n \in \mathbb{N}}$ is Markov with transition kernel $S_{k-1}$, then the random process $(X^{k-1}(n), X^k(n - \half), X^k(n))_{n \in \mathbb{N}}$ is also Markov with some transition kernel $Q_k$, since the evolution of the $k$th level only depends on the evolution of the $(k-1)$th level. Let $L_k$ be a Markov projection from the $(k-1)$th and $k$th level onto just the $k$th level. If we show that 
\begin{equation}\label{PRInt}
L_kQ_k = S_kL_k
\end{equation}
then the intertwining property of \cite{PR}, Theorem 2, will imply that the projection to the $k$th level is Markov with kernel $S_k$.

The explicit expression for $L_k$ is not hard to write. Define  $L_{k}$  from $ \mathcal{W}^{(2)}_k$ to $ \mathcal{W}_{k-1}\times  \mathcal{W}^{(2)}_k$ by 
\begin{align*}
L_k((z_0,y_0),(x,y,z)) &= 1_{(z_0,y_0)=(z,y)}\frac{s_{k-1}(x)}{s_k(y)}1_{x\prec y}
\end{align*}
By \eqref{Central}, the $L_k$ are Markov.

In order to prove \eqref{PRInt}, there also need to be explicit formulas for $Q_k$. In order to write these formulas, first introduce some notation. Let $\xi_1$ and $\xi_2$ be two independent geometric random variables with parameter $q$. For  $x\geq a \geq 0$, let $\overset{a\leftarrow}{P}(x,.)$ denote the law of the random variable
$$\max(a,x-\xi_1).$$
For $b\geq x\geq 0$, let $\overset{\rightarrow b}{P}(x,.)$ and $\overset{\rightarrow b}{R}(x,.)$ respectively denote the laws of the random variables
$$
\min(b,x+\xi_1)\,\textrm{ and } \,\min(b,\vert x+\xi_1-\xi_2\vert),
$$
For $x,y\in \R^2$ such that $x\leq y$ we let 
$$P(x,y)=(1-q)q^{y-x}.$$

With this notation in place, the description of the model implies the following explicit expression for $Q_k$.  For $(u,z,y),(x,z',y')\in  \mathcal{W}_{k-1} \times \mathcal{W}^{(2)}_k$ such that $u\prec y$ and $x\prec y'$  
\begin{align}\label{Qlodd} 
Q_{k}((u,z,y),(x,z',y'))=\sum_{v\in \N^{r-1}}S_{k-1}&(u,(v,x))\overset{\rightarrow v_{r-1}}{R}(y_r\wedge v_{r-1},y'_r) \nonumber\\
&\times \prod_{i=1}^{r-1}\overset{u_{i}\leftarrow}{P}(y_{i}\wedge v_{i-1},z'_{i})\prod_{i=1}^{r-1}\overset{ \rightarrow v_{i-1}}{P}(z'_{i}\vee x_{i},y'_{i}),
\end{align}
when $k=2r-1$ and
\begin{align}\label{Qleven}
Q_{k}((u,z,y),(x,z',y'))=\sum_{v\in \N^{r}}S_{k-1}&(u,(v,x)) \nonumber \\&\times \prod_{i=1}^{r} \overset{u_{i}\leftarrow}{P}(y_{i}\wedge v_{i-1},z'_{i})\prod_{i=1}^{r} \overset{ \rightarrow v_{i-1}}{P}(z'_{i}\vee x_{i},y'_{i}),
\end{align}
 when $k=2r$. In both cases $v_0=\infty$ and the sum runs over $v=(v_1,\dots,v_{r-1})\in \N^{r-1}$ (or $\mathbb{N}^r$) such that $v_i\in\{y'_{i+1},\dots,x_i\wedge z'_i\}$, for all $i.$ The notation can be depicted visually as 

\begin{center}
\begin{tabular}{ | c | c | c | c |}
  \hline 
  & time $n$ & time $n+1/2$ & time $n+1$ \\  \hline
  level $k$ & $y$ & $z'$ & $y'$ \\   \hline
  level $k-1$ & $u$ & $v$ & $x$ \\ \hline
\end{tabular}
\end{center}
Here is a description in words. For both $k=2r$ and $k=2r-1$, the $k$th level has $r$ particles, so after the $(k-1)$th  level evolves as $S_{k-1}$ (without dependence on what happens on the $k$th level), there is a $(r-1)$--fold double product corresponding to the left and right jumps of the $r-1$ particles away from the wall. For the particle closest to the wall, the evolution is as $R$ when $k$ is odd, and when $k$ is even the evolution fits into the previous $(r-1)$--fold double product. 
    
In order to show \eqref{PRInt}, there need to be explicit expressions and identities for these laws. The next lemma provides this.

\begin{lemma}\label{keyidentity}
For $(x,y,z)\in \N^3$ such that $0<z\le y$
\begin{align}\label{keyidentity1}
\sum_{u=0}^{z}  R(u,x)\overset{u\leftarrow}{P}(y,z)=  (1-q) q^{x\vee z+y-2z}.
\end{align}
For $(x,y,a)\in \N^3$ such that $a\le y$ and $y\le x$
\begin{align}\label{keyidentity2}
\sum_{u=a}^{y}q^u\overset{u\leftarrow}{P}(x,y)= q^{x-y}q^a.
\end{align}
For $(x,y,a)\in \N^3$ such that $y\le a$ and $x\le y$
\begin{align}\label{keyidentity3}
\sum_{v=y}^{a}q^{-v}\overset{\rightarrow v}{P}(x,y)= q^{y-x}q^{-a}.
\end{align}
For $y\in \N,y'\in \N^*$ such that $y'\le a$
\begin{align}\label{keyidentity4}
\sum_{v=y'}^{a}q^{v\vee y-2v}\overset{\rightarrow v}{R}(y\wedge v,y')= \frac{1}{1-q}q^{-a}R(y,y').
\end{align}
\end{lemma}  
\begin{proof} Note that \eqref{keyidentity2} and \eqref{keyidentity3} are precisely statements from Lemma 8.3 of \cite{D2}. Before showing the remaining identities are true, it is necessary to have formulas for these laws. The following two statements are from Lemma 8.2 of \cite{D1}: 

For $a,x,y\in \N$ such that $a\le y\le x$
\begin{align*}
\overset{a\leftarrow}{P}(x,y)&=\left\{
    \begin{array}{ll}
        (1-q)q^{x-y} & \mbox{if } a+1\le y \\
        q^{x-a} & \mbox{if } y=a.
    \end{array}
\right. 
\end{align*} 
For $b,x,y\in \N$ such that $b\ge y\ge x$
\begin{align*}
\overset{\rightarrow b}{P}(x,y)&=\left\{
    \begin{array}{ll}
        (1-q)q^{y-x} & \mbox{if } y\le b-1 \\
        q^{b-x} & \mbox{if } y=b.
    \end{array}
\right. 
\end{align*}

This next formula follows from direct computation. For $b,x,y\in \N$ such that $b\ge y, x$
\begin{align*}
\overset{\rightarrow b}{R}(x,y)&=\left\{
    \begin{array}{ll}
      \frac{1-q}{1+q}(q^{\vert y-x\vert }+q^{x+y+1}) & \mbox{if }  y\leq b-1,  \\
      \\
          \frac{1}{1+q}q^b(q^{-x }+q^{x+1}) & \mbox{if }  y=b, \\
          \\
    \end{array}
\right.
\end{align*}

So that
\begin{align*}
&\sum_{u=0}^z R(u,x)\stackrel{u \leftarrow}{P}(y,z)\\
&= \frac{1-q}{1+q}q^{y-z} \left( \sum_{u=0}^{z-1} (q^{u+x+1}+q^{\left| u-x\right|})(1-q) + (q^{z+x+1}+q^{\left| z-x\right|}) \right)\\
&=
\begin{cases}\frac{1-q}{1+q}q^{y-z} (q^{x+1}(1-q^z)+q^{x-z+1}-q^{x+1}+q^{z+x+1}+q^{x-z}), x\geq z\\
\frac{1-q}{1+q}q^{y-z} (q^{x+1}(1-q^z)-q^{x-1}+q+1-q^{z-x}+q^{z+x+1}+q^{ z-x}), x<z
\end{cases}
\end{align*}
which simplifies to $ (1-q) q^{x\vee z+y-2z} $.

Furthermore, 
\begin{align*}
&\sum_{v=y'}^{a}q^{v\vee y-2v}\overset{\rightarrow v}{R}(y\wedge v,y') = q^{y'\vee y-2y'} \ \overset{\rightarrow y'}{R}(y\wedge y',y') + \sum_{v=y'+1}^{a}q^{v\vee y-2v}\overset{\rightarrow v}{R}(y\wedge v,y') \\
&= q^{y' \vee y-2y'} \frac{1}{1+q} q^{y'}(q^{-y \wedge y'}+q^{y\wedge y'+1}) + \sum_{v=y'+1}^{a}q^{v\vee y-2v} \frac{1-q}{1+q} \left(q^{\left| y' - y\wedge v \right| }+q^{y\wedge v+y'+1}\right) \\
&= 
\begin{cases}
 \frac{1}{1+q} (q^{-y}+q^{y+1}) + \displaystyle\sum_{v=y'+1}^{a}q^{-v} \frac{1-q}{1+q} \left(q^{y' - y }+q^{y+y'+1}\right),  y \leq y' \leq a\\
 q^{y-2y'} \frac{1}{1+q} (1+q^{ 2y'+1}) + \displaystyle\sum_{v=y'+1}^{a}q^{y-2v} \frac{1-q}{1+q} \left(q^{v -y'   }+q^{ v+y'+1}\right), y' \leq a \leq y \\
  q^{y-2y'} \frac{1}{1+q} (1+q^{ 2y'+1}) + \displaystyle\sum_{v=y'+1}^{y}q^{ y-2v} \frac{1-q}{1+q} \left(q^{ v - y'  }+q^{ v+y'+1}\right) \\
 \quad \quad + \displaystyle\sum_{v=y+1}^{a}q^{-v} \frac{1-q}{1+q} \left(q^{ y-y'}+q^{y+y'+1}\right) , y' \leq y \leq a\\
\end{cases}
\end{align*}
Note that in each case, the summation over $v$ is of the form $\sum_v q^{-v}$, and in all cases simplifies to  $\frac{1}{1-q}q^{-a}R(y,y').$
\end{proof}

Now show that \eqref{PRInt} is true.
For $(z,y)\in \mathcal{W}^{(2)}_k$, $(x,z',y')\in \mathcal{W}_{k-1}\times \mathcal{W}_k^{(2)}$ such that $x\prec y'$,
\begin{align*}
L_{k}Q_{k}((z,y),(x,z',y'))&=\sum_{u\in \mathcal{W}_{k-1}}L_{k}((z,y),(u,z,y))Q_{k}((u,z,y),(x,z',y')).\end{align*}
Assume for now that $k=2r$. Then $L_kQ_k$ is equal to 
\begin{align*}
&\sum_{(u,v)\in \N^r\times \N^{r-1} } \frac{s_{k-1}(x)}{s_k(y)} (1-q)^{2r-2} R(u_r,x_r)q^{\sum_{i=1}^{r-1}(x_i+u_i-2v_i)}\nonumber\\&\quad \quad \quad\quad \quad \quad\quad \quad\quad \quad\quad \times P(z'_{1}\vee x_{1},y'_{1})\prod_{i=1}^{r} \overset{u_{i}\leftarrow}{P}(y_{i}\wedge v_{i-1},z'_{i})\prod_{i=2}^{r}  \overset{ \rightarrow v_{i-1}}{P}(z'_{i}\vee x_{i},y'_{i}).
\end{align*}
where the sum runs over $(u,v)\in \N^r\times \N^{r-1}$ such that $u_r\in\{0,\dots,z_r'\}$,  $v_i\in\{y'_{i+1},\dots,x_i\wedge z'_i\}$, $u_i\in \{v_i\vee y_{i+1},\dots,z'_i\}$, for $i\in\{1,\dots,r-1\}$.
Thus $L_kQ_k$ equals
\begin{align*}
&\sum_{v\in \N^{r-1}} \frac{s_{k-1}(x)}{s_k(y)} (1-q)^{2r-2}q^{\sum_{i=1}^{r-1}x_i}P(z'_{1}\vee x_{1},y'_{1}) \prod_{i=2}^{r}  q^{-2v_{i-1}}\overset{ \rightarrow v_{i-1}}{P}(z'_{i}\vee x_{i},y'_{i})\\&\quad \quad \quad\quad \quad  \times \sum_{ u\in \N^r }   ( (1-q)^{2r-2} R(u_r,x_r) \prod_{i=1}^{r} q^{u_i}\overset{u_{i}\leftarrow}{P}(y_{i}\wedge v_{i-1},z'_{i}) .
\end{align*}
Now evaluate the sum over $u$ and $v$. For each fixed $v$ the sum over $u$ is equal to
$$\sum_{u_r=0}^{z_r'} R(u_r,x_r)\overset{u_{r}\leftarrow}{P}(y_{r}\wedge v_{r-1},z'_{r})  \prod_{i=1}^{r-1}\sum_{u_i=v_i\vee y_{i+1}}^{z_i'}q^{u_i}\overset{u_{i}\leftarrow}{P}(y_{i}\wedge v_{i-1},z'_{i}).$$
Now, identities (\ref{keyidentity1}) and (\ref{keyidentity2}) of Lemma  \ref{keyidentity}  imply that the  sum over $u$ equals
$$
q^{x_r\vee z'_r+y_r\wedge v_{r-1}-2z_{r}'}(1-q)\prod_{i=1}^{r-1}q^{y_{i}\wedge v_{i-1}-z_{i}'+v_i\vee y_{i+1}},
$$
i.e.
$$
 q^{x_r\vee z'_r+y_r-2z'_r+\sum_{i=1}^{r-1}y_i+v_i-z_i'}(1-q).
$$
Thus
\begin{align*}
L_{2r}Q_{2r}((z,y),(x,z',y')) = & \frac{s_{k-1}(x)}{s_k(y)} (1-q)^{2r-1}  q^{x_r\vee z'_r+y_r-2z'_r+\sum_{i=1}^{r-1}y_i-z_i'}q^{\sum_{i=1}^{r-1}x_i} \\
&\quad \quad \quad   \times P(z'_{1}\vee x_{1},y'_{1})\prod_{i=2}^{r}  \sum_{v_{i-1}=y_{i}}^{x_{i-1}\vee z'_{i-1}} q^{-v_{i-1}}\overset{ \rightarrow v_{i-1}}{P}(z'_{i}\vee x_{i},y'_{i}).
\end{align*}
Identity (\ref{keyidentity3}) of Lemma \ref{keyidentity} gives that 
\begin{align*}
 \prod_{i=2}^{r}  \sum_{v_{i-1}=y_{i}}^{x_{i-1}\vee z'_{i-1}} q^{-v_{i-1}}\overset{ \rightarrow v_{i-1}}{P}(z'_{i}\vee x_{i},y'_{i})&=\prod_{i=2}^{r}q^{y'_i-z'_i\vee x_i-x_{i-1}\wedge z_{i-1}'}\\
&=q^{y_r'-z_r'\vee x_r-x_1\wedge z'_1}q^{\sum_{i=2}^{r-1}y'_i-x_i-z_i'},
\end{align*}
which implies 
 \begin{align*}
L_{2r}Q_{2r}((z,y),(x,z',y'))= \frac{s_{k-1}(x)}{s_k(y)} (1-q)^{2r} q^{\sum_{i=1}^{r}y_i+y'_i-2z_i'},
\end{align*}
which is quickly seen to be equal to $S_{2r}L_{2r}$. This finishes the proof when $k=2r$.

Similarly when $k=2r-1,$ 
 \begin{align*}
L_{2r}Q_{2r}((z,y),(x,z',y'))& =\sum_{u,v\in \N^{r-1}}\frac{s_{k-1}(x)}{s_{k}(y)}q^{\sum_{i=1}^{r-1}x_i-2v_i}\overset{\rightarrow v_{r-1}}{R}(y_r\wedge v_{r-1},y'_r) \\&   \quad \quad \quad \times \prod_{i=1}^{r-1}q^{u_i}\overset{u_{i}\leftarrow}{P}(y_{i}\wedge v_{i-1},z'_{i})\prod_{i=1}^{r-1}\overset{ \rightarrow v_{i-1}}{P}(z'_{i}\vee x_{i},y'_{i}),
\end{align*}
where the sum runs over $(u,v)\in \N^{r-1}\times \N^{r-1}$ such that   $v_i\in\{y'_{i+1},\dots,x_i\wedge z'_i\}$, $u_i\in \{v_i\vee y_{i+1},\dots,z'_i\}$, for $i\in\{1,\dots,r-1\}$. 
The rest of the calculations are similar, using identities (\ref{keyidentity2}), (\ref{keyidentity3}) and (\ref{keyidentity4}) of Lemma \ref{keyidentity}.  Therefore \eqref{PRInt} is true and the proof of Theorem \ref{MarkovProj} is done.

\end{proof}

\section{Asymptotics}\label{Asymp} The interacting particle system from \cite{C} is a determinantal point process. In general, a determinantal point processes on a discrete space $\mathcal{S}$ is uniquely characterized by an object called a correlation kernel, which is a function on $\mathcal{S}\times \mathcal{S}$. In \cite{C}, the asymptotics were calculated for Plancherel representations of $Sp(\infty)$. Here, we find the asymptotics for the pure alpha representations.

By Theorem 1.2 of \cite{C}, the correlation kernel at integer times $T$ is given by
\begin{align*}
K_T((s,k),(t,m)) &= \frac{2^{a_k+1/2}}{\pi} \frac{1}{2\pi i}\int_{-1}^1 \oint \frac{\phi(x)^T}{\phi(u)^T} J_{s,a_k}(x) J_{t,a_m}(u) \frac{(1-x)^{r_k}}{(1-u)^{r_m}} \frac{(1-x)^{a_k} (1+x)^{1/2}}{(x-u)}dudx \\
&+ 1_{k\geq m}\frac{2^{a_k+1/2}}{\pi} \int_{-1}^1 J_{s,a_k}(x) J_{t,a_m}(x) (1-x)^{r_k-r_m+a_k}(1+x)^{1/2}dx
\end{align*}
where $(s,k),(t,m) \in \mathbb{N}\times \mathbb{Z}_+$ and recall that $\phi(x)$ is the function from Proposition \ref{Same}. If $\phi(x)$ is replaced with $e^{(x-1)}$ and $T$ is allowed to take any nonnegative value, then $K$ becomes the correlation kernel corresponding to the Plancherel characters.

\subsection{Symmetric Pearcey}
Define the \textit{symmetric Pearcey kernel} $\mathcal{K}$ on $\R_+\times\R$ as follows (see Theorem 1.5 of \cite{C}). Let
\begin{multline*}\label{GaussianLikeKernel}
\mathcal{K}((\nu_1,\eta_1),(\nu_2,\eta_2))= \\
\frac{\sqrt{2}}{2\pi^2 i} \int_{-i\infty}^{i\infty}\int_0^{\infty}\exp\left( \frac{u^2-x^2}{8} + \frac{\eta_2 u - \eta_1 x}{2}\right) \sin(\nu_1 \sqrt{2x}) \sin(\nu_2 \sqrt{2u}) \frac{dxdu}{\sqrt{u}(u-x)}\\
+\frac{1_{\eta_2<\eta_1}}{\sqrt{\pi(\eta_1-\eta_2)}}\left(\exp\left(\frac{(\nu_1+\nu_2)^2}{\eta_2-\eta_1}\right)+\exp\left(\frac{(\nu_1-\nu_2)^2}{\eta_2-\eta_1}\right)\right).
\end{multline*}
By substituting $x\mapsto cx$ and $u\mapsto cu$,
\begin{multline*}
\mathcal{K}((\nu_1,\eta_1),(\nu_2,\eta_2))= \\
\frac{\sqrt{2}}{2\pi^2 i} \int_{-i\infty}^{i\infty}\int_0^{\infty}\exp\left( \frac{c^2u^2-c^2x^2}{8} + \frac{c\eta_2 u - c\eta_1 x}{2}\right) \sin(\nu_1 \sqrt{2cx}) \sin(\nu_2 \sqrt{2cu}) \frac{\sqrt{c} dxdu}{\sqrt{u}(u-x)}\\
+\frac{1_{\eta_2<\eta_1}}{\sqrt{\pi(\eta_1-\eta_2)}}\left(\exp\left(\frac{(\nu_1+\nu_2)^2}{\eta_2-\eta_1}\right)+\exp\left(\frac{(\nu_1-\nu_2)^2}{\eta_2-\eta_1}\right)\right).
\end{multline*}

\begin{theorem}
Let $c_{\alpha}$ be the constant $(1+\alpha)^{-2}(\alpha(2+\alpha))^{1/2}$, where $\alpha=2q/(1-q)$. Let $s_1$ and $s_2$ depend on $N$ in such a way that $s_i/N^{1/4}\rightarrow \nu_i c_{\alpha}^{1/2}>0$ as $N\rightarrow\infty$. Let $T$ and $r_1,r_2$ also depend on $N$ in such a way that $T/N\rightarrow 1$ and $(r_j- (1-(1+\alpha)^{-2})N)/\sqrt{N}\rightarrow c_{\alpha} \eta_j$. Then setting $k_j=2r_j+a_j-1/2$,
\[
(-2)^{r_2-r_1}(-1)^{s_1-s_2}2^{a_{n_2}-a_{n_1}} N^{1/4}{c_{\alpha}^{-1/2}}K((s_1,k_1),(s_2,k_2))\rightarrow \mathcal{K}((\nu_1,\eta_1),(\nu_2,\eta_2)).
\]
\end{theorem}
\begin{proof}
Since the proof is almost identical to the proof of Theorem 5.8 from \cite{BK} and Theorem 1.5 from \cite{C}, some of the details will be omitted. 
If $x' = N^{1/2} (x+1)$ and $u' = N^{1/2}(u+1)$, then by an (unnumbered) equation on page 41 of \cite{C},
$$
\frac{ (1-x)^{a_k}(1+x)^{1/2}}{x-u}dudx  \sim N^{-3/4} \frac{2^{a_k} \sqrt{x'}}{x' - u'}\cdot du'dx'.
$$
and if $s \sim N^{1/4}\nu$ then
$$
N^{1/4} (-1)^s J_{s,a}(x) \sim \frac{\sin [\nu \sqrt{2x'}] }{2^{a}\sqrt{x'}}
$$
To analyze the other terms in the integrand, define
\[
A(z)=\log\phi_{\alpha}(z)+(1-(1+\alpha)^{-2})\log(z-1),
\]
with asymptotic expansion
\[
A(z)-A(-1)=-\frac{\alpha(2+\alpha)}{8(1+\alpha)^4} (z+1)^2+O((z+1)^3).
\]
Thus the expression in $x$ becomes
\begin{multline*}
\exp\left( N(A(x)-A(-1)) + c_{\alpha}\eta\sqrt{N}(\log(1-x)-\log 2) \right) \\
\approx \exp\left( -\frac{c_{\alpha}^2}{8}(z+1)^2(x')^2 + c_{\alpha}\eta\sqrt{N}(\log(2-N^{-1/2}x')-\log 2)\right) \approx \exp\left( - \frac{c_{\alpha}^2}{8} (x')^2 - \frac{1}{2}c_{\alpha}\eta \cdot x' \right) 
\end{multline*}

\end{proof}

Note that after taking the determinant, the conjugating factors $(-2)^{r_2-r_1}(-1)^{s_1-s_2}2^{a_{n_2}-a_{n_1}}$ have no affect. 

\subsection{Discrete Jacobi}

For $-1< u < 1$ and $a_1,a_2=\pm\half$, define the \textit{discrete Jacobi kernel}
$L(r_1,a_1,s_1,r_2,a_2,s_2,b;u)$ as follows. If $2r_1 + a_1  \geq 2r_2 + a_2$, then
$$
L(r_1,a_1,s_1,r_2,a_2,s_2,b;u)
=\frac{2^{a_1+1/2}}{\pi}\int_u^1 J_{s_1,a_1}(x)   J_{s_2,a_2}(x)  (x-1)^{r_1 - r_2} (1-x)^{a_1}(1+x)^{b}dx.
$$
If $ 2r_1 + a_1  < 2r_2 + a_2$, then
$$
L(r_1,a_1,s_1,r_2,a_2,s_2,b;u)
=-  \frac{2^{a_1+1/2}}{\pi} \int_{-1}^u
J_{s_1,a_1}(x)   J_{s_2,a_2}(x)  (x-1)^{r_1 - r_2} (1-x)^{a_1}(1+x)^{b}dx.
$$
Note that $L$ only depends on $r_1,r_2$ through their difference $r_1-r_2$.
\begin{theorem}
Let $T$ depend on $N$ in such a way that $T/N\rightarrow t$. Let $r_1,\ldots,r_l$ depend on $N$ in such a way that $r_i/N\rightarrow l$ and their differences $r_i-r_j$ are fixed finite constants. Here, $t,l>0$. Fix $s_1,s_2,\ldots,s_l$ to be finite constants. Let 
\[
\theta=1+\frac{2l}{(l-t)(2\alpha+\alpha^2)}, \quad  \alpha=\frac{2q}{1-q}
\]
Then setting $k_j=2r_j+a_j-1/2$,
\begin{multline*}
\lim_{N\rightarrow\infty}\det[K_T((s_i,k_i),(s_j,k_j))]_{i,j=1}^l \\
=\begin{cases}
1,\ \ &l\geq (1-(1+\alpha)^{-2})t\\
\det[L(r_i,a_i,s_i,r_j,a_j,s_j  , 1/2 ;\theta)]_{i,j=1}^l,\ \ &l<(1-(1+\alpha)^{-2})t
\end{cases}
\end{multline*}
\end{theorem}
\begin{proof}
The proof of Theorem 4.1 of \cite{K} carries over here. The only difference is the parameters in the Jacobi polynomials, but these have no effect in the asymptotics.
\end{proof}

\bibliographystyle{plain}

\begin{center}
\begin{figure}
\caption{The top figure shows left jumps and the bottom figure shows right jumps. A yellow arrow means that the particle has been pushed by a particle below it. A green arrow means that the particle has jumped by itself. A red line means that the particle has been blocked by a particle below. 
\quad
In the table, keep in mind that $\xi^k_{(k+1)/2}(n+1/2)$ actually correspond to left jumps, but occur at the same time as the right jumps.
}
\label{Jumping}
\includegraphics[height=2in]{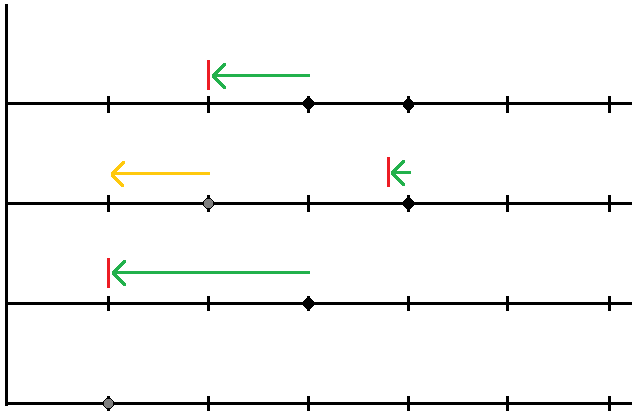}

    \renewcommand\arraystretch{1.33}
    \begin{tabular}{ | l | l | l | l | l | }
    \hline
    $\boldsymbol{\tilde{X}}(n)$ & Left Jumps & $\boldsymbol{\tilde{X}}(n+\half)$ & Right jumps  & $\boldsymbol{\tilde{X}}(n+1)$ \\ \hline
    $\tilde{X}^1_1(n)=1$ &                                     & $\tilde{X}^1_1(n+\half)=1$ & $\xi^1_1(n+\half)=1$ & $\tilde{X}^1_1(n+1)=3$\\  
			    &					    &					 & $\xi^1_1(n+1)=3$ 	 & 	 			\\ \hline
    $\tilde{X}^2_1(n)=3$ & $\xi^2_1(n+\half)=3$ & $\tilde{X}^2_1(n+\half)=1$ & $\xi^2_1(n+1)=1$ 	 & $\tilde{X}^2_1(n+1)=4$ \\ \hline
    $\tilde{X}^3_2(n)=2$ &                                     & $\tilde{X}^3_2(n+\half)=1$ & $\xi^3_2(n+\half)=2$ & $\tilde{X}^3_2(n+1)=0$\\      
			    &					    &					 & $\xi^3_2(n+1)=0$ 	 & 				  \\ \hline
    $\tilde{X}^3_1(n)=4$ & $\xi^3_1(n+\half)=1$ & $\tilde{X}^3_1(n+\half)=4$ & $\xi^3_1(n+1)=0$ 	 & $\tilde{X}^3_1(n+1)=5$  \\    \hline
    $\tilde{X}^4_2(n)=3$ & $\xi^4_2(n+\half)=2$ & $\tilde{X}^4_2(n+\half)=2$ & $\xi^4_2(n+1)=2$ 	 & $\tilde{X}^4_2(n+1)=3$  \\    \hline
    $\tilde{X}^4_1(n)=4$ & $\xi^4_1(n+\half)=0$ & $\tilde{X}^1_1(n+\half)=4$ & $\xi^4_1(n+1)=1$ 	 & $\tilde{X}^4_1(n+1)=6$  \\    \hline
    \end{tabular}
    
\includegraphics[height=2in]{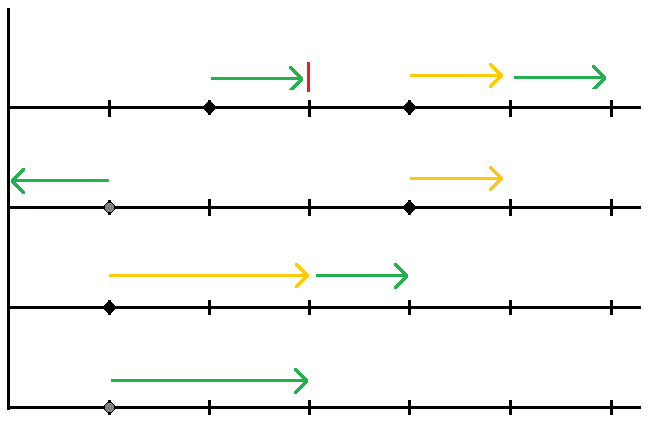}
\end{figure}
\end{center}

\end{document}